\documentclass{amsart}
\usepackage{amssymb}
\usepackage{amsfonts}
\usepackage{amssymb}
\usepackage{amsmath}
\usepackage{amsthm}
\usepackage{enumerate}
\usepackage{tabularx}
\usepackage{centernot}
\usepackage{mathtools}
\usepackage{stmaryrd}
\usepackage{amsthm,amssymb}
\usepackage{etoolbox}
\usepackage{url}
\usepackage{tikz}
\usepackage{amssymb}
\usetikzlibrary{matrix}
\usepackage{tikz-cd}
\usepackage{tikz}
\usepackage{marginnote}
\definecolor{mygray}{gray}{0.85}
\usepackage[backgroundcolor=mygray,colorinlistoftodos,prependcaption,textsize=small]{todonotes}
\usepackage{xargs}                      

\renewcommand{\leq}{\leqslant}
\renewcommand{\geq}{\geqslant}

\makeatletter
\def\subsection{\@startsection{subsection}{3}%
  \z@{.5\linespacing\@plus.7\linespacing}{.3\linespacing}%
  {\bfseries\centering}}
\makeatother

\makeatletter
\def\subsubsection{\@startsection{subsubsection}{3}%
  \z@{.5\linespacing\@plus.7\linespacing}{.3\linespacing}%
  {\centering}}
\makeatother

\makeatletter
\def\myfnt{\ifx\protect\@typeset@protect\expandafter\footnote\else\expandafter\@gobble\fi}
\makeatother

\newtheorem{theorem}{Theorem}[section]

\newtheorem{corollary}[theorem]{Corollary}
\newtheorem{definition}[theorem]{Definition}
\newtheorem{lemma}[theorem]{Lemma}
\newtheorem{proposition}[theorem]{Proposition}

\newtheorem{fact}[theorem]{Fact}

\newtheorem{remark}[theorem]{Remark}
\newtheorem{notation}[theorem]{Notation}

\newtheorem{convention}[theorem]{Convention}

\newcounter{claimcounter}
\numberwithin{claimcounter}{theorem}

\newcommand{\pureindep}[1][]{%
  \mathrel{
    \mathop{
      \vcenter{
        \hbox{\oalign{\noalign{\kern-.3ex}\hfil$\vert$\hfil\cr
              \noalign{\kern-.7ex}
              $\smile$\cr\noalign{\kern-.3ex}}}
      }
    }\displaylimits_{#1}
  }
}

\begin{document}
%%%%%%%%%%%%%%%%%%

\begin{abstract} 
We prove that every quasi-Hopfian finitely presented structure $A$ has a $d$-$\Sigma_2$ Scott sentence, and that if in addition $A$ is computable and $Aut(A)$ satisfies a natural computable condition, then $A$ has a computable $d$-$\Sigma_2$ Scott sentence. This unifies several known results on Scott sentences of finitely presented structures and it is used to prove that other not previously considered algebraic structures of interest have computable $d$-$\Sigma_2$ Scott sentences. 
In particular, we show that every right-angled Coxeter group of finite rank has a computable $d$-$\Sigma_2$ Scott sentence, as well as any strongly rigid Coxeter group of finite rank.
Finally, we show that the free projective plane of rank $4$ has a computable $d$-$\Sigma_2$ Scott sentence, thus exhibiting a natural example where the assumption of quasi-Hopfianity is used (since this structure is not Hopfian).
\end{abstract}

\title[Computable Scott Sentences for Finitely Presented Structures]{Computable Scott Sentences for Quasi-Hopfian Finitely Presented Structures}
%\thanks{The second author was supported by European Research Council grant 338821.}

%\author{Tapani Hyttinen}
%\address{Department of Mathematics and Statistics,  University of Helsinki, Finland}

\author{Gianluca Paolini}
\address{Department of Mathematics ``Giuseppe Peano'', University of Torino, Italy.}

\date{\today}
\maketitle

\section{Introduction}

	Scott proved that for every countable $L$-structure $A$ there is an $L$-sentence $\Theta_A$ of the logic $\mathfrak{L}_{\omega_1, \omega}$ such that for any countable $L$-structure $B$, $B \models \Theta_A$ if and only if $A \cong B$. Recently, there have been a number of papers (such\footnote{This list does not intend to be complete and we are sorry for omissions.} as \cite{ho, knight1, knight2, trainor1, trainor2}) which dealt with the problem of determination of the syntactic complexity of {\em optimal} Scott sentences for a given finitely generated structure, i.e. Scott sentences which are best as possible with respect to a given syntactic stratification of $\mathfrak{L}_{\omega_1, \omega}$.
	
	In particular, said results dealt with the following stratification of the formulas of $\mathfrak{L}_{\omega_1, \omega}$, a classification which is central in {\em computable model theory}.  We say:
	\begin{enumerate}[(1)]
	\item $\varphi(\bar{x})$ is computable $\Pi_0$ and computable $\Sigma_0$ if it is finitary quantifier-free; 
	\item For an ordinal (resp. a computable ordinal) $\alpha >0$:
	\begin{enumerate}[(2.1)]
	\item $\varphi(\bar{x})$ is $\Sigma_\alpha$ (resp. computable $\Sigma_\alpha$) if it is a disjunction (resp. a computably enumerable disjunction) of formulas of the form $\exists\bar{y}\psi(\bar{x}, \bar{y})$, where $\psi$ is a $\Pi_\beta$-formula (resp. computable $\Pi_\beta$-formula) for some $\beta < \alpha$;
\end{enumerate}
	\begin{enumerate}[(2.2)]
	\item $\varphi(\bar{x})$ is $\Pi_\alpha$ (resp. computable $\Pi_\alpha$) if it is a conjunction (resp. a computably enumerable conjunction) of formulas of the form $\forall\bar{y}\psi(\bar{x}, \bar{y})$, where $\psi$ is a $\Sigma_\beta$-formula (resp. computable $\Sigma_\beta$-formula) for some $\beta < \alpha$;
\end{enumerate}
	\begin{enumerate}[(2.3)]
	\item $\varphi(\bar{x})$ is $d$-$\Sigma_\alpha$ if it is a conjunction of a $\Sigma_\alpha$-formula and a $\Pi_\alpha$-formula.
\end{enumerate}\end{enumerate}

	As remarked in \cite{knight2} every finitely generated structure has a $\Sigma_3$ Scott sentence. On the other hand, it has been known for a while that many finitely generated structures of interest have a $d$-$\Sigma_2$ Scott sentence (or even a computable $d$-$\Sigma_2$ Scott sentence), such as for example finitely generated abelian groups \cite{calvert, knight2}, free groups of finite rank \cite{knight1}, and the infinite dihedral group \cite{knight2}. This motivated a whole program in computable model theory toward the identification of dividing lines for optimal Scott sentences of finitely generated structures. Among the many results we mention the resolution in the negative in \cite{trainor1} of an important open problem: does every finitely generated group has a $d$-$\Sigma_2$ Scott sentence? Furthermore, in \cite{ho} and \cite{trainor2} many new examples of finitely generated structures with $d$-$\Sigma_2$ (resp. computable $d$-$\Sigma_2$) Scott sentences were exhibited, \mbox{among which rings, fields, modules, etc.} %On the other hand, at the best of our knowledge there are no 
	
	In this work we try to identify some common abstract properties of a finitely presented structure $A$ ensuring that $A$ has a d-$\Sigma_2$ Scott sentence (resp. a computable d-$\Sigma_2$ Scott sentence). The main ingredient of our general approach is the following weaker version of Hopfianity,  which we refer to as {\em quasi-Hopfianity}.

	\begin{definition} Let $A$ be a finitely generated structure. We say that $A$ is quasi-Hopfian if there exists a finite generating set $\bar{a}$ of $A$ such that whenever $f: A \rightarrow A$ is a surjective homomorphism of $A$ which is injective on $\bar{a}$ we have that $f \in Aut(A)$.
\end{definition}

	Clearly, every Hopfian structure is quasi-Hopfian, but for example, as proved in \cite{johnson, sandler_hopfian}, free projective planes are quasi-Hopfian but not Hopfian. The utility of the notion of Hopfianity in the study of optimal Scott sentences of finitely generated structures is already well-known, for examples in \cite{trainor2} it is proved that finitely presented Hopfian groups have d-$\Sigma_2$ Scott sentences. Our contribution to the subject is twofold, on one hand we show that the weaker notion of quasi-Hopfinaity implies d-$\Sigma_2$ Scott sentences for {\em any finitely presented structure} $A$. On the other hand, and more interestingly, we define a computable condition on $Aut(A)$ imposing that a quasi-Hopfian computable structure $A$ has a {\em computable} d-$\Sigma_2$ Scott sentence. This computable condition is often implicitly verified in the study of group of automorphisms of finitely presented structures, and it is intended to establish a bridge between the algebraic study of $Aut(A)$ and the study of optimal Scott sentences for $A$. The last part of the paper focuses on applications. Firstly, we show that our general criterion covers the case of free groups of finite rank \cite{knight1}, free abelian groups of finite rank \cite{knight2}, and the infinite dihedral group \cite{knight2}. Secondly, we use our methods to prove the existence of computable d-$\Sigma_2$ Scott sentence for a large subclass of a class of finitely presented structures of central interest in group theory which, at the best of our knowledge, has not yet been considered in computable model theory: {\em Coxeter groups of finite rank}. Finally, we prove that the free projective plane of rank $4$ (which, as already observed, is not Hopfian) has a computable d-$\Sigma_2$-Scott sentence, thus exhibiting a \mbox{natural example where quasi-Hopfianity is used.}
	
	\medskip
	
	We now state our results:

	\begin{theorem}\label{main_theorem} Let $A = \langle a_1, ..., a_n \rangle_A$ be a quasi-Hopfian finitely presented structure. Then the $Aut(A)$-orbit of $(a_1, ..., a_n)$ in $A$ is $\Pi_1$-definable, and so $A$ has a $d$-$\Sigma_2$ Scott sentence (by \cite{alvir}). Suppose further that $A$ is computable and that there is a finite $X \subseteq Aut(A)$ and a computable $F: \omega^n \rightarrow \omega$ such that $Aut(A) = \langle X \rangle_{Aut(A)}$ and:
\begin{equation}\tag{$\star$}\text{for every $\alpha \in Aut(A)$, $lg_X(\alpha) \leq F(lg_{S}(\alpha(a_1)), ..., lg_{S}(\alpha(a_n)))$},
\end{equation}
where $S = \{ a_1, ..., a_n \}$, $lg_X(\alpha)$ is computed in $Aut(A)$ and $lg_{S}(\alpha(s_i))$ is computed in $A$.
%	such that the following conditions hold:
%\begin{enumerate}[(i)]
%	\item $X$ generates $Aut(A)$;
%	\item for every $\alpha \in Aut(A)$, $lg_X(\alpha) \leq F(lg_S(\alpha(s_1)), ..., lg_S(\alpha(s_n)))$.
%\end{enumerate} 
Then the $Aut(A)$-orbit of $(a_1, ..., a_n)$ in $A$ is definable by a {\em computable} $\Pi_1$-formula, and so $A$ has a {\em computable} $d$-$\Sigma_2$ Scott sentence (by \cite{alvir}).
\end{theorem}

	\begin{corollary}[\cite{knight1, knight2}]\label{free_groups_corollary} Free groups of finite rank, free abelian groups of finite rank, and the infinite dihedral group have computable $d$-$\Sigma_2$ Scott sentences. 
\end{corollary}

	\begin{corollary}\label{strongly_rigid_corollary} Let $G$ be a computable quasi-Hopfian finitely presented group and suppose that $Inn(G)$ has finite index in $Aut(G)$. Then $G$ has a computable $d$-$\Sigma_2$ Scott sentence. Thus, every strongly rigid Coxeter group of finite rank has a computable $d$-$\Sigma_2$ Scott sentence. In particular, every strongly $2$-spherical Coxeter group of finite rank has a computable $d$-$\Sigma_2$ Scott sentence, as well as every Coxeter group which acts effectively, properly and cocompactly on the affine \mbox{or hyperbolic
plane.}
\end{corollary}

	In relation to the corollary above, we wish to observe that the strongly rigid Coxeter groups of finite rank have been characterized in \cite{muller}, as a result of a joint effort involving various authors, and that the two specific cases mentioned in the statement of the corollary are just particular cases of this general classification.

	\begin{corollary}\label{Artin_Coxeter} Let $G$ be a finite graph product of primary cyclic groups. Then $G$ has a computable $d$-$\Sigma_2$ Scott sentence. In particular, if $G$ is a right-angled Coxeter group of finite rank, then $G$ has a computable $d$-$\Sigma_2$ Scott sentence. %In particular, if $G$ be a free (resp. free abelian) group of finite rank, then $G$ has a computable $d$-$\Sigma_2$ Scott sentence.
\end{corollary}

	%Corollary~\ref{Artin_Coxeter} implies the already known analogous results for free groups of finite rank \cite{knight1}, free abelian groups of finite rank \cite{knight2}, and the infinite dihedral group \cite{knight2}. %At the best of our knowledge a common framework for the just mentioned results was not present in the literature.

%	\begin{corollary} Let $G$ be a free (resp. free abelian) group of finite rank, then $G$ has a computable $d$-$\Sigma_2$ Scott sentence.
%\end{corollary}

We conjecture that the methods from \cite{laurence_artin} combined with our general results yields that every right-angled Artin group of finite rank also has a has a computable $d$-$\Sigma_2$ Scott sentence, but this is out of the scope of the present paper.% as it required a good deal of familiarity with the work of \cite{laurence_artin}, which we do not have.

	\begin{corollary}\label{cor_free_planes} Let $\pi^4$ be the free projective plane of rank $4$ (cf. \cite{hall}). Then $\pi^4$ is computable and has a computable $d$-$\Sigma_2$ Scott sentence.
\end{corollary}

%	\begin{corollary} Free semigroups of finite rank have computable $d$-$\Sigma_2$ Scott sentence.
%\end{corollary}

%	\begin{corollary} Let $G$ be a finitely generated right-angled Coxeter group (resp. right-angled Artin group). Then $G$ has a computable $d$-$\Sigma_2$ Scott sentence.
%\end{corollary}

%	\begin{corollary}[{\cite{knight1, knight2}}] Let $G$ be a free (resp. free abelian) group of finite rank. Then $G$ has a computable $d$-$\Sigma_2$ Scott sentence.
%\end{corollary}

	In Section~\ref{sec_proof} we introduce the necessary notation and then prove Theorem~\ref{main_theorem} and Corollaries~\ref{free_groups_corollary}~and~\ref{strongly_rigid_corollary}. In Section~\ref{sec_Coxeter} we introduce Coxeter groups and graph products of primary cyclic groups and prove what is needed to establish Corollary~\ref{Artin_Coxeter}. In Section~\ref{sec_planes} we introduce free projective planes and prove Corollary~\ref{cor_free_planes}.

\section{Proof of Main Theorem}\label{sec_proof}

	Our definition of finitely presented structure is standard, so we write $A = \langle \bar{a} \mid \varphi_1(\bar{a}), ..., \varphi_n(\bar{a}) \rangle$, where, for all $i \in [1, n]$, the formulas $\varphi_i(\bar{a})$ are assumed to be atomic $L$-formulas, for details see e.g. \cite[Section~9.2]{hodges} where this is explained and justified with care. Concerning the notions of length of an $L$-term and of length of an element $a$ of an $L$-structure $A$ with respect to a generating set $X$ for $A$, denoted as $lg_X(a)$, any reasonable definition (e.g. \cite[Chapter~1]{hodges}) makes our theorems true and so we prefer to remain vague. On the other hand, when dealing with groups or other particular structures, where the exact notion we choose might be relevant for the statements of the corresponding results, we use the notion of length established in that area of research (most notably we will do this for group theory).
	
	\begin{notation} Let $L$ be a finite language and $A$ a finitely generated $L$-structure. For the rest of this section we will assume that $A$ is finitely presented and we will fix one such presentation, and denote it as $A = \langle \bar{a} \mid \varphi_1(\bar{a}), ..., \varphi_n(\bar{a}) \rangle$, where, for all $i \in [1, n]$, the formulas $\varphi_i(\bar{a})$ are assumed to be atomic $L$-formulas.
\end{notation}

	\begin{notation}\label{isolation_notation} Let $A = \langle \bar{a} \mid \varphi_1(\bar{a}), ..., \varphi_n(\bar{a}) \rangle$, with $\bar{a} = (a_1, ..., a_m)$, all the $a_i$'s distinct and $\varphi_1(\bar{x}), ..., \varphi_n(\bar{x})$ a sequence of positive atomic formulas specifying a presentation of $A$ in the generators $\bar{a}$. Let $\psi(\bar{x})$ be the following formula:
	$$\bigwedge_{i \in [1, n]} \varphi_i(\bar{x}) \wedge \bigwedge_{i \neq j \in [1, n]} x_i \neq x_j.$$
Let then $X_*$ be the collection of $m$-tuples $\bar{b}$ of distinct element of $A$ such that:
\begin{enumerate}[(i)]
	\item $A \models \psi(\bar{b})$;
	\item $\langle \bar{b} \rangle_A \neq A$.
\end{enumerate}
For every $\bar{b} \in X_*$ fix terms $t_{(\bar{b}, 1)}(\bar{x}), ..., t_{(\bar{b}, m)}(\bar{x})$ s.t.
$A \models \bigwedge_{i \in [1, m]} b_i = t_{(\bar{b}, i)}(\bar{a})$.
\end{notation}

	\begin{remark} In the context of Notation~\ref{isolation_notation}, notice that if $A \models \psi(\bar{b})$, then the map $\bar{a} \mapsto \bar{b}$ is injective and it extends uniquely to an homomorphism of $A$.% (since $A = \langle \bar{a} \mid \varphi_1(\bar{a}), ..., \varphi_n(\bar{a}) \rangle$).
\end{remark}

	\begin{lemma}\label{crucial_lemma} In the context of Notation~\ref{isolation_notation}, so that $A = \langle \bar{a} \rangle_A$, assume that $A$ is quasi-Hopfian, then the $Aut(A)$-orbit of $\bar{a}$ is defined in $A$ by the $\Pi_1$-formula $\Theta(\bar{x})$:
	$$\psi(\bar{x}) \wedge \bigwedge_{\bar{b} \in X_*} \forall \bar{y} \neg (\psi(\bar{y}) \wedge \bigwedge_{i \in [1, n]} x_i = t_{(\bar{b}, i)}(\bar{y})).$$
\end{lemma}

	\begin{proof} We want to show that $\bar{b} = (b_1, ..., b_n) \models \Theta(\bar{x})$ if and only if there exists $\alpha \in Aut(A)$ such that $\alpha(\bar{a}) = \bar{b}$.
Concerning the implication ``left-to-right'', suppose that $A \models \Theta(\bar{b})$. It suffices to show that $\bar{b} \notin X_*$, since then the map $f: A \rightarrow A$ which maps $a_i \mapsto b_i$ is on one hand surjective (recall the definition of $X_*$) and on the other hand injective on $\bar{a}$ (recall that $\psi(\bar{x})$ is a conjunct of $\Theta(\bar{x})$), and so, by quasi-Hopfianity, $f$ is an automorphism of $A$. Suppose that $\bar{b} \in X_*$, then we have:
$$A \models \psi(\bar{a}) \wedge \bigwedge_{i \in [1, n]} b_i = t_{(\bar{b}, i)}(\bar{a}),$$
contradicting the fact that $\bar{b} \models \Theta(\bar{x})$.
Concerning the implication ``right-to-left'', let $\alpha \in Aut(A)$ and let $\bar{b} = \alpha(\bar{a})$, we want to show that $A \models \Theta(\bar{b})$. Clearly, $A \models \psi(\bar{b})$. For the sake of contradiction, suppose that for some $\bar{c} \in X_*$ we have that:
$$A \models \exists \bar{y} (\psi(\bar{y}) \wedge \bigwedge_{i \in [1, n]} b_i = t_{(\bar{c}, i)}(\bar{y})).$$
Then there exists $\bar{d} \in A$ such that:
\begin{equation}\label{eq} A \models \psi(\bar{d}) \wedge \bigwedge_{i \in [1, n]} b_i = t_{(\bar{c}, i)}(\bar{d}).
\end{equation}
But then, since $A = \langle \bar{a} \rangle_A$, $\alpha \in Aut(A)$ and $\bar{b} = \alpha(\bar{a})$, by the second conjunct of (\ref{eq}) we have that $\langle \bar{d} \rangle_A = A$, and so by the quasi-Hopfianity of $A$ we have that:
$$\beta: a_i \mapsto d_i \in Aut(W).$$
Furthermore:
%$$ t_{(\bar{c}, i)}(\bar{d})) : i \in [1, n]\} = \{b_1, ..., b_n \}$$
%is a basis of $W$, and so:
$$\gamma: d_i \mapsto t_{(\bar{c}, i)}(\bar{d}) = b_i \in Aut(W).$$
Hence, we have:
$$\begin{array}{rcl}
(\beta^{-1} \circ \gamma \circ \beta)(a_i) & = & (\beta^{-1} \circ \gamma)(d_i)\\
 & = & \beta^{-1}(t_{(\bar{c}, i)}(d_1, ..., d_{n}))\\
 & = & t_{(\bar{c}, i)}(\beta^{-1}(d_1), ..., \beta^{-1}(d_{n}))) \\
 & = & t_{(\bar{c}, i)}(a_1, ..., a_{n}) \\
 & = & c_i. \\
\end{array}$$
and so the map $a_i \mapsto c_i = t_{(\bar{c}, i)}(\bar{a}) \in Aut(A)$, contradicting the fact that $\bar{c} \in X_*$.
\end{proof}

	\begin{proof}[Proof of Theorem~\ref{main_theorem}] The first claim is by Lemma~\ref{crucial_lemma}. We show that under the additional assumptions we have that the $Aut(A)$-orbit of $(a_1, ..., a_n)$ in $A$ is definable by a {\em computable} $\Pi_1$-formula. It suffices to show that in this case the conjunction:
$$\bigwedge_{\bar{b} \in X_*} \forall \bar{y} \neg (\bigwedge_{i \in [1, n]} \varphi_i(\bar{y}) \wedge \bigwedge_{i \in [1, n]} x_i = t_{(\bar{b}, i)}(\bar{y})).$$
is computably enumerable. To this extent, let $X_+ = \{ \bar{b} \in A^n : A \models \psi(\bar{b}) \}$. We exhibit an algorithmic procedure which takes as input tuples $\bar{b} \in X_+$ and answers $\bf{YES}$ if $\bar{b} \in X_*$ and answers $\bf{NO}$ if $\bar{b} \in X_+ \setminus X_{*}$. Let $\bar{b} \in X_+$ and fix terms $t_{(\bar{b}, 1)}(\bar{x}), ..., t_{(\bar{b}, n)}(\bar{x})$ such that $A \models \bigwedge_{i \in [1, n]} b_i = t_{(\bar{b}, i)}(\bar{a})$. Notice that as the language is finite there are only finitely many terms of each length. Since $A$ is computable we can assume without loss of generality that for every $i \in [1, n]$ we have that $lg_{S}(b_i) \leq lg(t_{(\bar{b}, i)}(\bar{x}))$ (recall that $S = \{a_1, ..., a_n \}$). Now, let
$k = F(lg_{S}(b_1), ..., lg_{S}(\alpha(b_n)))$
and enumerate all the elements $\alpha \in Aut(A)$ such that $\alpha = \alpha_1^{\pm 1} \circ \cdots \circ \alpha_m^{\pm 1}$, with $\alpha_i \in X$ and $m \leq k$, and call the resulting finite collection of automorphisms $B_0$. Then in order to decide if $\bar{b} \in X_*$ or not it suffices to check if for some $\beta \in B_0$ we have that: 
$$\beta(a_1) = t_{(\bar{b}, 1)}(\bar{a}) = b_1, ..., \beta(a_n) = t_{(\bar{b}, n)}(\bar{a}) = b_n,$$
and this is a computable task, since $A$ is assumed to be a computable structure.
\end{proof}

\begin{proof}[Proof of Corollary~\ref{strongly_rigid_corollary}] Let $G = \langle S\rangle_G$ be as in the assumption of the corollary. Let $\beta_1, ..., \beta_k$ be representatives of the cosets of $Inn(G)$ in $Aut(G)$ and let $\alpha_1, ..., \alpha_n$ be the inner automorphisms corresponding to the generators in $S$ (so that the automorphisms $\alpha_1, ..., \alpha_n$ generate $Inn(G)$). Then letting:
$$X = \{\alpha_1, ..., \alpha_n\} \cup \{\beta_1, ..., \beta_k\} \subseteq Aut(G),$$
and $F: \omega^{n} \rightarrow \omega$ be the following function\footnote{We are not interested here in optimal functions $F$ that make $(\star)$ true.}: $F(m_1, ..., m_n) = (\sum_{i \in [1, n]} m_i) + 1$, we have that condition $(\star)$ of Theorem~\ref{main_theorem} is verified for this choice of $X$ and $F$.

\smallskip

\noindent Concerning the claims about Coxeter groups, this is by Fact~\ref{fact1} and \cite{muller}, specifically in \cite[page 539, line -10]{muller} it is observed that the result of \cite{muller} imply that strongly $2$-spherical Coxeter groups are strongly rigid, and it is well-known that if a finitely gen. Coxeter group $W$ is strongly rigid, then $Inn(W)$ has finite index in $Aut(W)$.
\end{proof}

	\begin{proof}[Proof of Corollary~\ref{free_groups_corollary}] It is well-known that these groups are Hopfian and have solvable word problem. Thus, to conclude, it suffices to show that the assumptions of Theorem~\ref{main_theorem} are satisfied. Concerning the case of free groups, take as $X$ the set of Nielsen transformations and let $F(m_1, ..., m_n) = \sum_{i \in [1, n]} m_i$; then, as noted in the proof of \cite[Theorem~2.6]{knight2}, Nielsen proved that condition $(\star)$ of Theorem~\ref{main_theorem} is verified for this choice of $X$ and $F$. Concerning the case of free abelian groups, this follows from the fact that $Aut(\mathbb{Z}^n)$ is the set of $n \times n$ invertible $\mathbb{Z}$-matrices, and this group is generated by certain finitely many matrices (see e.g. \cite[Appendix~C]{elman}). The claim about the infinite dihedral group is by Corollary~\ref{Artin_Coxeter}.
\end{proof}

\section{Coxeter Groups and Graph Products of Groups}\label{sec_Coxeter}

	In this section we deal with applications to Coxeter groups, a class of groups that arises in a multitude of ways in several areas of mathematics, such as algebra \cite{humphreys}, geometry \cite{davis} and combinatorics \cite{brenti}. We now define what a Coxeter group is.

\begin{definition}[Coxeter groups]\label{def_Coxeter_groups} Let $S$ be a set. A matrix $m: S \times S \rightarrow \{1, 2, . . . , \infty \}$ is called a {\em Coxeter matrix} if it satisfies:
	\begin{enumerate}[(1)]
	\item $m(s, s') = m(s' , s)$;
	\item $m(s, s') = 1 \Leftrightarrow s = s'$.
	\end{enumerate}
%$$m(s, s') = m(s' , s);$$
%$$m(s, s') = 1 \Leftrightarrow s = s'.$$
For such a matrix, let $S^2_{*} = \{(s, s') \in S^2 : m(s, s' ) < \infty \}$. A Coxeter matrix $m$ determines
a group $W$ with presentation:
$$
\begin{cases} \text{Generators:} \; \;  S \\
				\text{Relations:} \; \;   (ss')^{m(s,s')} = e, \text{ for all } (s, s' ) \in S^2_{*}.
\end{cases} $$
A group with a presentation as above is called a Coxeter group, and the pair $(W, S)$ is a called a Coxeter system. The rank of the Coxeter system $(W, S)$ is $|S|$. The rank of the group $W$ is the rank of any Coxeter system $(W, S)$ with $|S|$ minimal. In this paper we are only interested in Coxeter groups of finite rank.
\end{definition}

	\begin{notation}\label{def_Coxeter_graph} In the context of Definition~\ref{def_Coxeter_groups}, the Coxeter matrix $m$ is often equivalently represented by a labeled graph $\Gamma$ whose node set is $S$ and whose edges are the pairs $\{s, s' \}$ such that $m(s, s') < \infty$, with label $m(s, s')$. Notice that some authors consider instead the graph $\Delta$ such that $s$ and $s'$ are adjacent iff $m(s, s ) \geq 3$. In order to try to avoid confusion we refer to the first graph as the Coxeter graph of $(W, S)$ (and usually denote it with the letter $\Gamma$), and to the second graph as the Coxeter diagram of $(W, S)$ (and usually denote it with the letter $\Delta$).
\end{notation}

\begin{definition}\label{def_irreducible} Let $(W, S)$ be a Coxeter system with Coxeter diagram $\Delta$ (recall Notation~\ref{def_Coxeter_graph}). We say that $(W, S)$ is irreducible if $\Delta$ is connected.
\end{definition}

	\begin{definition}[Right-angled Coxeter and Artin groups]\label{def_Artin_Coxeter} Let $m$ be a Coxeter matrix and let $W$ be the corresponding Coxeter group. We say that $W$ is right-angled if the matrix $m$ has values in the set $\{ 1, 2, \infty\}$. In this case the Coxeter graph $\Gamma$ associated to $m$ is simply thought as a graph (instead of a labeled graph), whith edges corresponding to the pairs $\{ s, s' \}$ such that $m(s, s') = 2$. A right-angled Artin group is defined as in the case of right-angled Coxeter groups with the omission in the defining presentation of the requirement that generators have order~$2$.
\end{definition}

	Artin groups will not play a role in the rest of the paper, we gave the definition of right-angled Artin groups in Definition~\ref{def_Artin_Coxeter} to give context to the conjecture made before Corollary~\ref{cor_free_planes}, i.e. that our methods might apply also to these structures.

	\begin{definition}[Strongly $2$-spherical Coxeter groups] Let $(W, S)$  be a Coxeter system of finite rank with Coxeter matrix $m$. We say that the Coxeter system $(W, S)$ is $2$-spherical if $m$ has only finite entries. We say that $(W, S)$ is strongly $2$-spherical if in addition $W$ is not finite and $(W, S)$ is irreducible. We say that the Coxeter group $W$ is $2$-spherical (resp. strongly $2$-spherical) if there is $S \subseteq W$ such that $(W, S)$ is a $2$-spherical (resp. strongly $2$-spherical) Coxeter system. 
\end{definition}

	It is a stadard fact that a finitely generated group is computable if and only if it has solvable word problem \cite{rabin} -- this is why we state Facts~\ref{fact1}~and~\ref{fact_hop_graph_product}.

\begin{fact}\label{fact1} Coxeter groups of finite rank are Hopfian and have \mbox{solvable word problem.}
\end{fact}

	\begin{proof} As well-known such groups are linear groups over the real numbers and thus residually finite, and in particular Hopfian. The solvability of the word problem is also well-known (first proved by Tits), see e.g. \cite[Section~3.4]{davis} for a reference.
\end{proof}

	\begin{definition} Let $(W, S)$  be a Coxeter system. We say that $W$ is a strongly rigid Coxeter group if for every $T \subseteq W$ such that $(W, T)$ is a Coxeter system there exists $w \in W$ such that $T = S^w$, where $S^w$ stands for $\{wsw^{-1} : s \in S \}$.
\end{definition}

\begin{definition}\label{def_cyclic_prod} Let $\Gamma = (V, E)$ be a graph and $\mathbf{p}: V \rightarrow \{ p^n : p \text{ prime, } n \geq 1 \}$ a graph coloring\footnote{The non-edges of $\Gamma$ may be equivalently interpreted as edges labelled $\infty$, in which case the graph $\Gamma$ is complete, but we chose not to use this convention in our presentation.}. We define a group $G(\Gamma, \mathbf{p})$ with the following presentation:
	$$ \langle V \mid a^{\mathbf{p}(a)} = 1, \; bc = cb : \mathbf{p}(a) \neq \infty \text{ and }  b E c \rangle.$$
We call groups of the form $G(\Gamma, \mathbf{p})$ graph products of primary cyclic groups.
\end{definition}

	\begin{convention}\label{convention_finite_graph} From now on all the graph products of primary cyclic groups $G(\Gamma, \mathbf{p})$ considered in this paper are assumed to be such that $\Gamma$ is finite.
\end{convention}

	\begin{fact}[\cite{green}]\label{fact_hop_graph_product} Graph products of primary cyclic groups $G(\Gamma, \mathbf{p})$ (with $\Gamma$ finite, cf. Convention~\ref{convention_finite_graph}) are Hopfian and have solvable word problem.
\end{fact}

%	\begin{proof} More generally, graph products of residually finite groups are residually finite.
%\end{proof}

	\begin{definition}\label{complete_subgroups} Let $G(\Gamma, \mathbf{p})$ be a graph product of primary cyclic groups. A subgroup of $G$ which is generated by the vertices of a maximal complete subgraph (a.k.a. a maximal clique) of $\Gamma$ is called a maximal complete subgroup.
\end{definition}

	\begin{notation} Let $G(\Gamma, \mathbf{p})$ be a graph product of primary cyclic groups. We denote by $Spe(G)$ the subgroup of $Aut(G)$ consisting of those automorphisms $\alpha$ of $G$ such that $\alpha(v)$ is a conjugate of $v$ for every $v \in \Gamma$. We denote by $F(\Gamma)$ the subgroup of $Aut(G)$ consisting of those automorphisms of $G$ which map each maximal complete subgroup of $G$ to a maximal complete subgroup of $G$ (cf. Definition~\ref{complete_subgroups}).
\end{notation}

	\begin{remark} Let $G = G(\Gamma, \mathbf{p})$ be a graph product of primary cyclic groups and denote by $G_{ab}$ the abelianization of $G$. Then $F(\Gamma)$ is isomorphic to the image of $Aut(G)$ under the natural map $Aut(G) \rightarrow Aut(G_{ab})$. In particular, $F(\Gamma)$ is finite.
\end{remark}

	\begin{fact}[{\cite[Theorem~1.2]{gutierrez}}]\label{semi-direct} Let $G = G(\Gamma, \mathbf{p})$ be a graph product of primary cyclic groups. Then $Aut(G)$ admits the following semi-direct product decomposition:
	$$Aut(G) = Spe(G) \rtimes F(\Gamma).$$
\end{fact}

	\begin{definition}[{\cite[Proposition~4.2, Definition~4.3]{laurence}}]\label{def_length_autos} Let $G(\Gamma, \mathbf{p})$ be a graph product of primary cyclic groups. For $\alpha \in Spe(G)$ with $\alpha(v) = w_v v w_v^{-1}$ and $w_v v w_v^{-1}$ a normal form,  we define the length of $\alpha$, denoted as $|\alpha|$, to be $\sum_{v \in \Gamma} lg_{\Gamma}(w_v)$. %In \cite[Proposition~4.2]{laurence} it is shown that this is well-defined.
\end{definition}

	\begin{definition}\label{partial_conj} Let $\Gamma$ be a graph, $v \in \Gamma$ and $C$ a union of connected components of $\Gamma \setminus N^*(v)$, where $N^*(v) = \left\{ v' \in \Gamma : v E_{\Gamma} v' \right\} \cup \{ v \}$. We define a map $\pi_{(s, C)}$ as:
$$\begin{cases} \pi_{(s, C)}(t) = sts \;\;\;\; \text{ if } t \in C \\
			  \pi_{(s, C)}(t) = t \;\;\;\;\;\;\; \text{ otherwise. }
\end{cases} $$
The maps of the form $\pi_{(s, C)}$ are called the {\em partial conjugations} of $G(\Gamma, \mathbf{p})$.
\end{definition}

\begin{fact}[\cite{gutierrez}]\label{partial_conj_fact} The partial conjugations are automorphisms of $G(\Gamma, \mathbf{p})$.
\end{fact}

	\begin{fact}[{\cite[Theorem~4.1]{laurence}}]\label{laurence} Let $G(\Gamma, \mathbf{p})$ be a graph product of primary cyclic groups. Then $Spe(G)$ is generated by the set $X$ of partial conjugations corresponding to the graph $\Gamma$. Further, if $\alpha \in Spe(G)$, then $lg_X(\alpha) \leq |\alpha|$ (cf. Definition~\ref{def_length_autos}).
\end{fact}

	\begin{proof}[Proof of Corollary~\ref{Artin_Coxeter}] Let $G(\Gamma, \mathbf{p})$ be a graph product of primary cyclic groups. Let $\alpha_1, ..., \alpha_n$ be a list of the partial conjugations corresponding to the graph $\Gamma$ and let $\beta_1, ..., \beta_k$ be a list of the elements of $F(\Gamma)$. Let also:
$$X = \{\alpha_1, ..., \alpha_n\} \cup \{\beta_1, ..., \beta_k\} \subseteq Aut(G),$$
and $F: \omega^{n} \rightarrow \omega$ be the following function\footnote{We are not interested here in optimal functions $F$ that make $(\star)$ true.}: $F(m_1, ..., m_n) = (\sum_{i \in [1, n]} m_i) + 1$. Then, by Fact~\ref{fact1}, we can apply Theorem~\ref{main_theorem}, and by Facts~\ref{semi-direct}~and~\ref{laurence} we have that $(\star)$ is verified for this choice of $X$ and $F$, and so we are done.
\end{proof}

\section{The Free Projective Plane of Rank $4$}\label{sec_planes}

\begin{definition}[{\cite{hall}}]\label{def_plane} A {\em partial plane} is a system of points and lines satisfying:
	\begin{enumerate}[(A)]
	\item through any two distinct points $p$ and $p'$ there is at most one line $p \vee p'$;
	\item any two distinct lines $\ell$ and $\ell'$ intersect in at most one point $\ell \wedge \ell'$.
\end{enumerate}
We say that a partial plane is a {\em projective plane} if in (A)-(B) above we replace ``at most'' with ``exactly one''. We say that a projective plane is non-degenerate if it contains a quadrangle, i.e. four points such that no three of them are collinear.
\end{definition}

	\begin{definition}[{\cite{hall}}]\label{free_extension} Given a partial plane $P$ we define a chain of partial planes $(P_n : n < \omega)$, by induction on $n < \omega$, as follows:
\newline $n = 0)$. Let $P_n = P$.
\newline $n = 2k +1)$. For every pair of distinct points $p, p' \in P_{2k}$ not joined by a line add a new line $p \vee p'$ to $P_{2k}$ incident \mbox{with only $p$ and $p'$. Let $P_n$ be the resulting plane.}
\newline $n = 2k >0)$. For every pair of parallel lines $\ell, \ell' \in P_{2k-1}$ add a new point $\ell \wedge \ell'$ to $P_{2k-1}$ incident \mbox{with only $\ell$ and $\ell'$. Let $P_n$ be the resulting plane.}
\newline We define the {\em free projective extension} of $P$ to be $F(P) : = \bigcup_{n < \omega} P_n$.
\end{definition}

	\begin{notation}\label{pi_n} Given $4 \leq n \leq \omega$, we let $\pi_0^n$ be the partial plane consisting of a line $\ell$, $n-2$ points on $\ell$ and $2$ points off of $\ell$, and we let $\pi^n = F(\pi^n_0)$. We refer to the plane $\pi^n$, for $4 \leq n \leq \omega$, as the free projective plane of rank $n$. Further, given $k < \omega$, we say that $x \in F(\pi^n_0) = \bigcup_{m < \omega} P_m$ is of stage $k$ if $x \in P_k \setminus P_{k-1}$.
%We say that a plane is free if it is isomorphic to $\pi^n$ for some $4 \leq n \leq \omega$.
\end{notation}

	\begin{notation}\label{notation_theory} Model-theoretically we consider projective planes $P$ as $L$-structures in a language $L = \{ 0, 1, S_1, S_2, I, \vee, \wedge \}$, where we let the following:
	 \begin{enumerate}[(i)]
	 \item $0$ and $1$ are constant symbols;
	 \item $S_1$ specifies the set of points of $P$ and $S_2$ specifies the set of lines of $P$;
	 \item $I$ is a symmetric binary relation specifying the point-line incidence relation;
	 \item the interpretation of $\wedge$ (intersection of lines) and $\vee$ (join of points) are extended naturally so that $(P, 0, 1, \wedge, \vee)$ becomes a modular geometric lattice. Explicitly, for $p, p' \in S_1$ and $\ell, \ell' \in S_2$ we let $p \wedge p' = 0$, $\ell \vee \ell' = 1$ and:
	 $$p \wedge \ell = \begin{cases} p \; \text{ if } \; p I \ell \\
	 							    0 \; \text{ otherwise;}
\end{cases}
     p \vee \ell = \begin{cases} \ell \; \text{ if } \; p I \ell \\
	 							    1 \; \text{ otherwise.}
\end{cases}$$
\end{enumerate}
\end{notation}

	\begin{remark}\label{collineation} Notice that under this choice of language $L$, if $P$ is a projective plane then $Aut(P)$ is the collineation group of $P$, i.e. the set of bijections of $P$ sending points to points, lines to lines and preserving the point-line incidence relation.
\end{remark}

	\begin{fact}\label{hopfian_fact} Let $4 \leq n < \omega$, then $\pi^n$ is finitely presented, quasi-Hopfian but not Hopfian.
\end{fact}

	\begin{proof} The fact that $\pi^n$ is finitely presented is clear. Concerning quasi-hopfianity, let $f: \pi^n \rightarrow \pi^n$ be a surjective homomorphism of $\pi^n$ which is injective on $\pi^n_0$, then, by \cite[Th.~3]{sandler_hopfian},  $f(\pi^n_0)$ generates $\pi^n$ freely, hence $f \in Aut(A)$. Finally, the fact that $\pi^n$ is not Hopfian is proved in \cite{johnson}.
\end{proof}

	\begin{proposition}\label{computable_fact} Let $4 \leq n < \omega$, then $\pi^n$ is computable.
\end{proposition}

	\begin{proof} In \cite{kobaev} it is proved that, for $4 \leq n < \omega$, $\pi^n$ is computable in a language specifying the set of points, the set of lines and the graph of the partial functions $\vee$ and $\wedge$. Furthermore, in \cite{nikitin} it is proved, under the same computable representation, that the incidence problem for $\pi^n$ is decidable. Thus, it is immediate to see that $\pi^n$ is computable also with respect to our choice of language (Notation~\ref{notation_theory}).
\end{proof}

	\begin{notation} Clearly for $n = 4$ we can consider $\pi^n_0$ as consisting simply of four points $A_1, A_2, B_1, B_2$. Let now $a_1 = (A_1 \vee A_2) \wedge (B_1 \vee B_2)$ and $a_2 = (A_1 \vee B_1) \wedge (A_2 \vee B_2)$, and consider the collineations (cf. Remark~\ref{collineation} for a definition of collineation) of $\pi^4$ detemined by the following assignments:
	$$\theta_1 = \left( \begin {array}{c} 
	  A_1 \mapsto A_2 \\%\noalign{\medskip}
	  A_2 \mapsto A_1 \\%\noalign{\medskip}
	  B_1 \mapsto B_1 \\%\noalign{\medskip}
	  B_2 \mapsto B_2 \\%\noalign{\medskip}
	 \end {array} \right), \quad
	 \theta_2 = \left( \begin {array}{c} 
	  A_1 \mapsto A_2 \\%\noalign{\medskip}
	  A_2 \mapsto B_1 \\%\noalign{\medskip}
	  B_1 \mapsto B_2 \\%\noalign{\medskip}
	  B_2 \mapsto A_1 \\%\noalign{\medskip}
	 \end {array} \right), \quad
	 \phi = \left( \begin {array}{c} 
	  A_1 \mapsto A_1 \\%\noalign{\medskip}
	  A_2 \mapsto a_1 \\%\noalign{\medskip}
	  B_1 \mapsto B_1 \\%\noalign{\medskip}
	  B_2 \mapsto a_2 \\%\noalign{\medskip}
	 \end {array} \right). $$
\end{notation}

\begin{fact}[{\cite{sandler}}] $Aut(\pi^4) = \langle \theta_1, \theta_2, \phi \rangle_{Aut(\pi^4)}$, where $\theta_1$ and $\theta_2$ generate a group $S_4$ isomorphic to the symmetric group on the four points $A_1, A_2, B_1, B_2$, while $\phi^2 = id_{\pi^4}$. Thus, any element $\alpha \in Aut(\pi^4)$ can be written as follows, where $P_{i_j} \in S_4$:
\begin{equation}\tag{$*$} \alpha = P_{i_1} \circ \phi \circ P_{i_2} \circ \cdots \circ P_{i_n} \circ \phi \circ P_{i_{n+1}}.
\end{equation}
We say that $\alpha \in Aut(\pi^4)$ is of length (at most) $n$ if $\alpha$ can be written in form $(*)$ above and in $(*)$ there are $n$ occurrences of the collineation $\phi$.
\end{fact}

	\begin{definition} We say that $\alpha \in Aut(\pi^4)$ is of stage $n < \omega$ if for every $C \in \{A_1, A_2, B_1, B_2\}$ we have that $\alpha(C)$ is of stage $ \leq n$ (in the sense of Notation~\ref{pi_n}).
\end{definition}

	\begin{fact}[{\cite{sandler}}]\label{sandler_fact}
Every collineation $\alpha \in Aut(\pi^4)$ of stage $n$ is of length $\leq n$.
\end{fact}

	\begin{proof}[Proof of Corollary~\ref{cor_free_planes}] Let $X = \{\theta_1, \theta_2, \phi\}$ and $F: \omega^{n} \rightarrow \omega$ be the following function\footnote{We are not interested here in optimal functions $F$ that make $(\star)$ true.}: $F(m_1, ..., m_n) = \sum_{i \in [1, n]} 2(m_i + 1)$, then by Facts~\ref{hopfian_fact}~and~\ref{sandler_fact} and Proposition~\ref{computable_fact} we can apply Theorem~\ref{main_theorem} and $(\star)$ is verified \mbox{for this choice of $X$ and $F$.}
\end{proof}

\end{document}